\documentclass[12pt]{article}

\usepackage{amsmath}
\usepackage{amsthm}
\usepackage{amssymb}
\usepackage{mathptmx}
\usepackage{bbding}
\usepackage{graphicx}
\usepackage{multirow,makecell}

\title{\bf Isolated singularities of conformal hyperbolic metrics}
\author{Yu Feng$^1$, Yiqian Shi$^2$ and Bin Xu$^3$}
\usepackage{fancyhdr}

\def\nd{\noindent}
\newtheorem{thm}{Theorem}[section]
\newtheorem{lem}{Lemma}[section]

\newtheorem{defi}{Definition}[section]
\newtheorem{rem}{Remark}[section]

\begin{document}

\maketitle

\nd{\small  $^{1,2,3}$Wu Wen-Tsun Key Laboratory of Mathematics, USTC, Chinese Academy of Sciences\\
$^{1,2,3}$School of Mathematical Sciences, University of Science and Technology of China.\\
No. 96 Jinzhai Road, Hefei, Anhui Province\  230026\  P. R. China.}

\nd {\small $^{1}$ yuf@mail.ustc.edu.cn \quad  $^{2}$ yqshi@ustc.edu.cn\quad  $^3$\Envelope bxu@ustc.edu.cn}
\par\vskip0.5cm

\nd {\small {\bf Abstract:}  J. Nitsche proved that an isolated singularity of a conformal hyperbolic metric is either a conical singularity or a cusp one ({\small \emph{Math. Z. } {\bf 68} (1957) 316-324}).
We prove by developing map that there exists a complex coordinate $z$ centered at the singularity where the metric has the expression of either $\displaystyle{\frac{4\alpha^2\vert z \vert^{2\alpha-2}}{(1-\vert z \vert ^{2\alpha})^2}\vert \mathrm{d} z \vert^2}$ with $\alpha>0$ or $\displaystyle{\vert z \vert ^{-2}\big(\ln|z|\big)^{-2}|dz|^{2}}$.\\

\nd{\bf Keywords.} hyperbolic metric, conical singularity, cusp singularity, developing map\\

\nd {\bf 2010 Mathematics Subject Classification.} Primary 51M10; Secondary 34M35

\section{Introduction}

\paragraph{} Let $\Sigma$ be a compact Riemann surface and $\textup{D}=\sum\limits_{i=1}^n(\theta_{i}-1)p_{i}$ a ${\Bbb R}$-divisor on $\Sigma$ such that $0\leq \theta_i\not=1$. We call that a smooth conformal metric $\mathrm{d}\sigma^2$ on
$\Sigma\backslash {\rm supp}\, D=\Sigma\backslash \{p_1,\cdots, p_n\}$
{\it represents the divisor} $D$  if the following condition is satisfied:
\begin{itemize}
                 \item \ If $\theta_i>0$, then $\mathrm{d} \sigma^2$\ {\it has a conical singularity at $p_{i}$ with cone angle\ $2\pi\theta_{i}>0$}. That is, in a neighborhood $U$ of  $p_{i}$, $\mathrm{d} \sigma^2=e^{2u}\vert \mathrm{d} z \vert^2$, where $z$ is a complex coordinate of $U$  with\ $z(p_{i})=0$ and\ $u-(\theta_i-1)\ln \vert z \vert$\ extends to a continuous function in $U$.

         \item \ If $\theta_i=0$, then $\mathrm{d} \sigma^2$\ {\it has a cusp singularity at\ $p_{i}$}. That is, in a neighborhood $V$ of\ $p_{i}$,\ $\mathrm{d} \sigma^2=e^{2u}\vert \mathrm{d} z \vert^2$, where\ $z$\ is a complex coordinate of $V$ with\ $z(p_{i})=0$ and\ $u+\ln\, \vert z \vert+\ln\,(-\ln\, \vert z \vert)$\ extends to a continuous function in $V$.

 \end{itemize}
\noindent We observe that the metric $\mathrm{d} \sigma^2$ has finite (infinite) diameter near a conical (cusp) singularity.
J. Nitsche \cite{N57}, M. Heins \cite{HE62} and A. Yamada \cite{Ya88} proved that an isolated singularity of a conformal hyperbolic metric must be either a conical singularity or a cusp one. By the Gauss-Bonnet formula, if\ $\mathrm{d} \sigma^2$\ is a conformal hyperbolic metric representing $\textup{D}$, then $\chi(\Sigma)+\sum\limits_{i=1}^n(\theta_{i}-1)<0$, where $\chi(\Sigma)$ is the Euler number of $\Sigma$. M. Heins \cite{HE62} proved by S-K metric the following existence theorem.

\begin{thm}
\label{thm:rat}
There exists a unique conformal hyperbolic metric representing a $\mathbb{R}$-divisor\ $\textup{D}=\sum\limits_{i=1}^n(\theta_{i}-1)p_{i}$ with $0\leq \theta_{i}\not=1$ on a compact Riemann surface $\Sigma$ if and only if\ $\chi(\Sigma)+\sum\limits_{i=1}^n(\theta_{i}-1)<0$.
\end{thm}

After Heins' work \cite{HE62}, both McOwen \cite{MO88} and Troyanov \cite{Tr91}, who were unaware of \cite{HE62} apparently,  proved the theorem for hyperbolic metrics with only conical singularities by using different PDE methods. In this manuscript, we will give the explicit expressions of conformal hyperbolic metrics near their isolated singularities.

\begin{thm}[Main Theorem]
\label{thm:Belyi}
Let\ $\mathrm{d} \sigma^2$\ be a conformal hyperbolic metric on the punctured disk\ $\bigtriangleup^{\ast}=\{{w\in\mathbb{C}|0<\vert w \vert<1}\}$. If\ $\mathrm{d} \sigma^2$\ has a conical singularity at\ $w=0$\ with the angle\ $2\pi\alpha>0$, then there exists a complex coordinate\ $z$ on\ $\Delta_{\varepsilon}=\{{w\in\mathbb{C}|\vert w \vert<\varepsilon}\}$\ for some\ $\varepsilon>0$ with\ $z(0)=0$ such that
$$ \mathrm{d} \sigma^2|_{\Delta_{\varepsilon}}=\frac{4\alpha^2\vert z \vert^{2\alpha-2}}{\big(1-\vert z \vert ^{2\alpha}\big)^2}\vert \mathrm{d} z \vert^2.$$
Moreover,\ $z$ is unique up to replacement by\ $\lambda z$ where\ $\vert \lambda \vert=1$. If\ $\mathrm{d} \sigma^2$\ has a cusp singularity at\ $w=0$, then there exists a complex coordinate\ $z$ on\ $\Delta_{\varepsilon}=\{{w\in\mathbb{C}|\vert w \vert<\varepsilon}\}$\ for some\ $\varepsilon>0$ with\ $z(0)=0$ such that
$$ \mathrm{d} \sigma^2|_{\Delta_{\varepsilon}}=\vert z \vert ^{-2}\big(\ln|z|\big)^{-2}|dz|^{2}.$$
Moreover,\ $z$ is unique up to replacement by\ $\lambda z$ where\ $\vert \lambda \vert=1$.

\end{thm}

The main theorem paves a way for the further study of global geometry of conformal hyperbolic metrics with singularities on compact Riemann surfaces.

We organize  the left part of the manuscript as follows. In section \ref{sec:rat}, we at first recall some elementary knowledge of hyperbolic geometry. Then we give the properties of developing maps and prove a hyperbolic version of Lemma 3.2 in \cite{CWWX15} for hyperbolic metrics (Lemma \ref{lem:NotEmp}). Section \ref{sec:Belyi} is the proof of the main theorem.

\vspace{0.5cm}

\section{Preliminaries}
\label{sec:rat}

\subsection{Elementary hyperbolic geometry}
\label{subsec:Rie}

We will work with any of the Poincar{\' e} disk model\ $\left(\mathbb{D}=\{z\in \mathbb{C}:\vert z \vert<1\},\ \frac{4\vert dz \vert^{2}}{(1-\vert z \vert^{2})^{2}}\right)$ and the upper half-plane model \ $\left(\mathbb{H}=\{z\in \mathbb{C}:\textup{Im z}>0\},\ \frac{\vert dz \vert^{2}}{(\Im z)^{2}}\right)$ of the two dimensional hyperbolic space form at our convenience.
We denote
$$\textup{PSU(1,1)}=\{z\longmapsto\frac{az+b}{\overline{b}z+\overline{a}}:\ a,\ b\in \mathbb{C},\ \vert a \vert^2-\vert b \vert^2=1\}$$
and
 $$\textup{PSL}(2,\mathbb{R})=\{z\longmapsto\frac{az+b}{cz+d}:\ a,\ b,\ c,\ d\in \mathbb{R},\ ad-bc=1\}$$
the group of all orientation-preserving isometries of\ $\mathbb{D}$ and\ $\mathbb{H}$, respectively, and denote by\ $\textup{I}(\mathbb{D})$ the isometry group of $\mathbb{D}$.

\begin{defi}
\label{defi:res}\rm{\cite[p.136]{R06}}.
  {\rm If\ $\mathfrak{L},\ \mathfrak{K}\in\textup{I}(\mathbb{D})$ and\ $\mathfrak{L}$ is not the identity map of ${\Bbb D}$, then\ $\mathfrak{L}$ has a fixed point in\ $\overline{\mathbb{D}}$. The transformation\ $\mathfrak{L}$ is said to be

\begin{itemize}
         \item {\it elliptic} if\ $\mathfrak{L}$ fixes a point of\ $\mathbb{D}$;

         \item {\it parabolic} if\ $\mathfrak{L}$ fixes no point of\ $\mathbb{D}$ and fixes a unique point of \ $\partial{\mathbb{D}}$;

         \item {\it hyperbolic} if\ $\mathfrak{L}$ fixes no point of\ $\mathbb{D}$ and fixes two points of \ $\partial{\mathbb{D}}$.

 \end{itemize} }
\end{defi}

\noindent Note that\ $\mathfrak{L}$ is elliptic, parabolic, or hyperbolic if and only if\ $\mathfrak{K}\circ\mathfrak{L}\circ\mathfrak{K^{-1}}$ is elliptic, parabolic, or hyperbolic, respectively. Thus, being elliptic, parabolic, or hyperbolic depends only on the conjugacy class of $\mathfrak{L}$ in\ $\textup{I}(\mathbb{D})$.
Elliptic, parabolic, or hyperbolic transformations of\ $\mathbb{H}$\ are defined in the similar manner as in the Poincar{\' e} disk model\ $\mathbb{D}$.

\begin{thm}
\label{thm:per}\rm{\cite[Theorem 2.31, p.67]{A05}}, \rm{\cite[p.172-173]{BE83}}.
  \begin{itemize}
         \item In the Poincar{\' e} disk model\ $\mathbb{D}$,  if\ $\mathfrak{L}\in \textup{PSU(1,1)}$ is elliptic, then there exists\ $\mathfrak{K}\in \textup{PSU(1,1)}$ such that\ $\mathfrak{K}\circ\mathfrak{L}\circ\mathfrak{K^{-1}}(z)=e^{i\theta}z$ for some real number\ $\theta$.

         \item In the upper half-plane model\ $\mathbb{H}$,  if\ $\mathfrak{L}\in \textup{PSL}(2,\mathbb{R})$ is parabolic, then there exists\ $\mathfrak{K}\in \textup{PSL}(2,\mathbb{R})$ such that\ $\mathfrak{K}\circ\mathfrak{L}\circ\mathfrak{K^{-1}}(z)=z+t$  for some real number\ $t$.

         \item In the upper half-plane model\ $\mathbb{H}$,  if\ $\mathfrak{L}\in \textup{PSL}(2,\mathbb{R})$ is hyperbolic, then there exists\ $\mathfrak{K}\in \textup{PSL}(2,\mathbb{R})$ such that\ $\mathfrak{K}\circ\mathfrak{L}\circ\mathfrak{K^{-1}}(z)=\lambda z$ for some positive real number\ $\lambda$.
  \end{itemize}
\end{thm}

\subsection{Developing map}
\label{subsec:three}

In the paper \cite{CWWX15}, Q. Chen, W. Wang, Y. Wu and the last author discussed developing maps of conical metrics with constant curvature one on a compact Riemann surface and their monodromy in detail. Actually these conclusions are also true in non-compact cases. We will give the corresponding case for conical hyperbolic metrics. In particular, we shall state the following two lemmas and omit their proof, which is similar as Lemmas 2.1-2 in \cite{CWWX15}. We need to prepare some notions at first.

A multi-valued locally univalent meromorphic function\ $H$\ on a Riemann surface\ $\Sigma$ is said to be {\it projective} if any two function elements\ $\mathfrak{H_{1}}, \mathfrak{H_{2}}$\ of\ $H$ near a point\ $p\in \Sigma$  are related by a fractional linear transformation\ $\mathfrak{L}\in \textup{PGL}(2,\mathbb{C})$, i.e.,\ $\mathfrak{H_{1}}=\mathfrak{L}\circ \mathfrak{H_{2}}$.

Let\ $\mathrm{d} \sigma^2$\ be a conformal hyperbolic metric on a Riemann surface\ $\Sigma$, not necessarily compact, representing the divisor\ $\textup{D}=\sum\limits_{i=1}^n(\theta_{i}-1)p_{i},\  0\leq \theta_{i}\not=1$. We call a projective function\ $F:\Sigma\backslash {\rm supp}\, D\longrightarrow \mathbb{D}$\ a {\it developing map} of the metric\ $\mathrm{d} \sigma^2$\ if\ $\mathrm{d} \sigma^2=F^{*}\mathrm{d} \sigma_{0}^2$, where\ $\mathrm{d} \sigma_{0}^2=\frac{4|\mathrm{d} z|^2 }{(1-|z|^2)^2}$\ is the hyperbolic metric on the unit disc\ $\mathbb{D}$.\\

\begin{lem}
\label{lem:pq}
 Let\ $\mathrm{d} \sigma^2$\ be a conformal hyperbolic metric on a compact Riemann surface\ $\Sigma$, representing the divisor $\textup{D}=\sum\limits_{i=1}^n(\theta_{i}-1)p_{i},\  0\leq \theta_{i}\not=1$. Then there exists a projective function\ $F$\ from\ $\Sigma\backslash {\rm supp}\, D$\ to the unit disc\ $\mathbb{D}$ such that the monodromy of\ $F$\ belongs to \textup{PSU(1,1)} and
\[
\mathrm{d} \sigma^2=F^{*}\mathrm{d} \sigma_{0}^2,
\]
where\ $\mathrm{d} \sigma_{0}^2=\frac{4|\mathrm{d} z|^2 }{(1-|z|^2)^2}$\ is the hyperbolic metric on\ $\mathbb{D}$.

\end{lem}

\begin{lem}
\label{lem:contr}
  Any two developing maps\ $F_{1}$,\ $F_{2}$ of the metric\ $\mathrm{d} \sigma^2$\ are related by a fractional linear transformation\ $\mathfrak{L}\in \textup{PSU}(1,1)$, i.e.,\ $F_{2}=\mathfrak{L}\circ F_{1}$. In particular, any two developing maps of\ $\mathrm{d} \sigma^2$\ have mutually conjugate monodromy in\ $\textup{{PSU}(1,1)}$. Then we call this conjugate class {\rm the monodromy of the metric\ $\mathrm{d} \sigma^2$}.

\end{lem}

The following lemma expresses the Schwarzian derivative of a developing map.\\

\begin{lem}{\rm{\cite[Theorem 1.2]{LLX17}}}.
  \label{lem:Transitive}
  Let\ $\mathrm{d} \sigma^2$\ be a conformal hyperbolic metric on a Riemann surface\ $\Sigma$, and suppose\ $\mathrm{d} \sigma^2$\ represents a divisor\ $\textup{D}=\sum\limits_{i=1}^n(\theta_{i}-1)p_{i},\  0\leq \theta_{i}\not=1$.
   Suppose that\ $F:\Sigma\backslash {\rm supp}\, D\longrightarrow \mathbb{D}$\ is a developing map of\ $\mathrm{d} \sigma^2$. Then the Schwarzian derivative\ $\{F,z\}=\frac{f^{'''}(z)}{f^{'}(z)}-\frac{3}{2}\Big(\frac{f^{''}(z)}{f^{'}(z)}\Big)^{2}$ of\ $F$\ equals
\[
\{F,z\}=\frac{1-\theta_{i}^{2}}{2z^{2}}+\frac{d_{i}}{z}+\phi_{i}(z)
\]
in a neighborhood\ $U_{i}$\ of\ $p_{i}$\ with complex coordinate\ $z$ and\ $z(p_{i})=0$, where the\ $d_{i}$\ are constants and the\ $\phi_{i}$\ are holomorphic functions in\ $U_{i}$, depending on the complex coordinate\ $z$.
\end{lem}

 $F$ is said to be {\it compatible with the divisor}\
 $\textup{D}=\sum\limits_{i=1}^n(\theta_{i}-1)p_{i},\  0\leq \theta_{i}\not=1$
 if the Schwarzian\ $\{F,z\}$ of\ $F$\ has the form above. We now go into local expressions of developing maps near conical or cusp singularities of conformal hyperbolic metrics in the following lemma, where for the sake of clarity we change the notion of $D$ a little.

\begin{lem}
\label{lem:NotEmp}
  Let\ $F:\Sigma\backslash {\rm supp}\, D\longrightarrow \mathbb{D}$ be a projective function compatible with the divisor
  $\textup{D}=\sum\limits_{i=1}^n(\alpha_{i}-1)p_{i}-\sum\limits_{j=1}^mq_{j}$ with $0<\alpha_j\not=1$ such that its monodromy lies in  $\textup{{PSU}(1,1)}$. Then there exists a neighborhood\ $U_{i}$ of\ $p_{i}$\ with complex coordinate\ $z$\ and\ $\mathfrak{L}_{i}\in \textup{PGL}(2,\mathbb{C})$ such that\ $z(p_{i})=0$\ and\ $G_{i}=\mathfrak{L}_{i}\circ F$\ has the form\ $G_{i}(z)=z^{\alpha_{i}}$; and there exists a neighborhood\ $U_{j}$ of\ $q_{j}$\ with complex coordinate\ $z$\ and\ $\mathfrak{L}_{j}\in \textup{PGL}(2,\mathbb{C})$, such that\ $z(q_{j})=0$\ and\ $G_{j}=\mathfrak{L}_{j}\circ F$\ has the form\ $G_{j}(z)=\log z$. Moreover the pullback metric $F^{*}\mathrm{d} \sigma_{0}^2$ on $\Sigma\backslash {\rm supp}\, D$ is a conformal hyperbolic metric representing $D$.

\end{lem}

\begin{proof}
  We first show that there exists a neighborhood\ $U_{i}$ of\ $p_{i}$\ with complex coordinate\ $z$\ and some\ $\mathfrak{L}_{i}\in \textup{PGL}(2,\mathbb{C})$ such that\ $z(p_{i})=0$\ and\ $G_{i}=\mathfrak{L}_{i}\circ F$\ has the form of\ $z^{\alpha_{i}}$. Since\ $F$ is compatible with\ $\textup{D}$, we could choose a neighborhood\ $U_{i}$ of\ $p_{i}$\ and a complex coordinate\ $x$\ on\ $U_{i}$\ such that\ $x(p_{i})=0$\ and
\[
                \{F,x\}=\frac{1-\alpha_{i}^{2}}{2x^{2}}+\frac{d_{i}}{x}+\phi_{i}(x),
\]
where\ $\phi_{i}(x)$\ is holomorphic in\ $U_{i}$.
By {\textup{[2]}}, in the neighborhood\ $U_{i}$\ there are two linearly independent solutions\ $u_{1}$\ and\ $u_{2}$\ of the equation
\begin{equation}
 \frac{d^{2}u}{dx^{2}}+\frac{1}{2}\left(\frac{1-\alpha_{i}^{2}}{2x^{2}}+\frac{d_{i}}{x}+\phi_{i}(x)\right)u=0
\end{equation}
with single-valued coefficient such that\[
  F(x)=u_{2}(x)/u_{1}(x).
\]
Define an operator\ $L_{i}:=x^{2}\frac{d^{2}}{dx^{2}}+q_{i}(x)$ with\ $q_{i}(x)=\left((1-\alpha_{i}^{2})/2+d_{i}x+x^{2}\phi_{i}(x)\right)/2$. Then both\ $u_{1}$\ and\ $u_{2}$\ are solutions of the equation\ $L_{i}u=0$. Note that the equation\ $\textup{(1)}$\ has a regular singularity at 0. We could apply the Frobenius method to solve it. The indicial equation
\[
             f(s)=s(s-1)+\frac{1-\alpha_{i}^{2}}{4}=0
\]
of the differential equation\ $L_{i}u=0$\ at\ $x=0$\ has roots\ $s_{1}=(1-\alpha_{i})/2$ and\ $s_{2}=(1+\alpha_{i})/2$, and\ $s_{2}-s_{1}=\alpha_{i}>0$. Let\ $\sum_{k=0}^\infty b_{k}x^{k}$ be the power series expansion of\ $q_{i}(x)$, where\ $b_{0}=\frac{1-\alpha_{i}^{2}}{4}$. Let\ $s$\ be a parameter. Then
\[
         u(s,x)=x^{s}\sum_{k=0}^\infty c_{k}(s)x^{k}\ (c_{0}(s)=1)
\]
is a solution of\ $L_{i}u=0$\ if and only if the equation
\[
         (\sharp)_{n}:f(s+n)c_{n}+R_{n}=0
\]
holds for all\ $n=0,1,2,\cdots$, where\ $R_{0}=0$, and, for\ $n>0$,
\[
       R_{n}=R_{n}(c_{1},\cdots,c_{n-1},s)=\sum_{i=0}^{n-1}c_{i}b_{n-i}.
\]
Since\ $f(s_{2}+n)\neq0$ for all\ $n\geq1$, we find that\ $ u(s_{2},x)$\ is a solution of the equation. Note that\ $c_{n}(s)$ is a rational function of\ $s$.
Suppose that\ $\alpha_i=s_{2}-s_{1}$\ is not an integer. Then, by the same reasoning,\ $u(s_{1},x)$\ is another solution, which is linearly independent of\ $u(s_{2},x)$.
Suppose that\ $m:=\alpha_i=s_{2}-s_{1}$\ is an integer\ $\geq2$ and\ $R_{m}=0$. In which case, we can solve the equation\ $(\sharp)_{n}$ for\ $s=s_{1}$ for all\ $n\geq1$ by choosing\ $c_{m}$ arbitrarily and obtain another solution\ $u(s_{1},x)$ linearly independent of\ $u(s_{2},x)$.
Using the same argument as \cite[Lemma 3.2]{CWWX15}, we could complete the proof for the above two cases.
We need different argument to rule out the case that $m=s_{2}-s_{1}$\ is an integer\ $\geq2$ and\ $R_{m}\neq 0$, whose details occupy the next paragraph.

Suppose that $m=s_{2}-s_{1}$\ is an integer\ $\geq2$ and\ $R_{m}\neq 0$. Then, by \cite[p.23]{Y87}, the following expression
\[
     u_{\log}:=u(s_{2},x)\log x+x^{s_{1}}v(x)
\]
is a logarithmic solution£¬ where
$$ v(x)=-\frac{f'(s_{2})}{R_{m}}\sum_{k=0}^\infty c_{k}(s_{1})x^{k}+x^{m}\sum_{k=0}^\infty c'_{k}(s_{2})x^{k}.$$
Note that\ $f'(s_{2})\neq 0$,\ $v(x)$\ is holomorphic near\ $0$ and\ $v(0)=-\frac{f'(s_{2})}{R_{m}}\neq 0$.
Thus\ $u_{\log}$,\ $u(s_{2},x)$\ are the two linearly independent solutions of\ $L_{i}u=0$\ near\ $0$. We have
$$ F(x)=\frac{au_{\log}+bu(s_{2},x)}{cu_{\log}+du(s_{2},x)}=:\begin {pmatrix} a & b \\ c & d \end {pmatrix}\circ \frac{u_{\log}}{u(s_{2},x)},$$
where
$$\begin {pmatrix} a & b \\ c & d \end {pmatrix}\in \textup{PGL}(2,\mathbb{C}),\  \  \frac{u_{\log}}{u(s_{2},x)}=\frac{u(s_{2},x)\log x+x^{s_{1}}v(x)}{u(s_{2},x)}.$$
Since the monodromy of\ $F$\ belongs to\ $\textup{PSU(1,1)}$, we can choose a function element$$\mathfrak{F}=\begin {pmatrix} a & b \\ c & d \end {pmatrix}\circ \frac{u_{\log}}{u(s_{2},x)}$$ in an open disk which is near $p_i$ and does not contain\ $p_{i}$. Then there exists\ $\begin {pmatrix} s & t \\ \overline{t} & \overline{s} \end {pmatrix}\in \textup{PSU}(1,1)$ such that
$$ \begin{aligned}
  &\  \  \begin {pmatrix} s & t \\ \overline{t} & \overline{s} \end {pmatrix}\circ \mathfrak{F}\\&=\begin {pmatrix} a & b \\ c & d \end {pmatrix}\circ \frac{e^{2\pi\sqrt{-1}s_{2}}u(s_{2},x)(\log x+2\pi\sqrt{-1})+e^{2\pi\sqrt{-1}s_{1}}x^{s_{1}}v}{e^{2\pi\sqrt{-1}s_{2}}u(s_{2},x)}\\&=\left\{\begin {pmatrix} a & b \\ c & d \end {pmatrix}\begin {pmatrix} 1 & 2\pi\sqrt{-1} \\ 0 & 1 \end {pmatrix}{\begin {pmatrix} a & b \\ c & d \end {pmatrix}}^{-1}\right\}\circ \mathfrak{F}.
\end{aligned} $$
We arrive at
$$ \pm \begin {pmatrix} s & t \\ \overline{t} & \overline{s} \end {pmatrix}=\begin {pmatrix} 1-ac\cdot 2\pi\sqrt{-1} & a^{2}\cdot 2\pi\sqrt{-1}  \\ -c^{2}\cdot 2\pi\sqrt{-1} & 1+ac\cdot 2\pi\sqrt{-1} \end {pmatrix},$$
therefore,\ $a\neq0$,\ $c\neq0$ and\ $|c|=|a|$.
Furthermore,
$$\frac{u_{\log}}{u(s_{2},x)}=\frac{u(s_{2},x)\log x+x^{s_{1}}v}{u(s_{2},x)}=\log x+x^{-m}\varphi,$$
where
$$\varphi=\frac{-\frac{f'(s_{2})}{R_{m}}\sum_{k=0}^\infty c_{k}(s_{1})x^{k}+x^{m}\sum_{k=0}^\infty c'_{k}(s_{2})x^{k}}{\sum_{k=0}^\infty c_{k}(s_{2})x^{k}}.$$
Note that\ $\varphi(x)$\ is holomorphic near\ $0$ and\ $\varphi(0)=-\frac{f'(s_{2})}{R_{m}}\neq 0$. Then, we have
$$ \begin{aligned}
   F(x)&=\frac{a(\log x+x^{-m}\varphi)+b}{c(\log x+x^{-m}\varphi)+d}\\&=\frac{a}{c}-\frac{1}{c}\cdot\frac{1}{c(\log x+x^{-m}\varphi)+d}
   \\&=\frac{a}{c}\left[1-\frac{1}{ac(x^{m}\log x+\varphi)+ad\cdot x^{m}}\cdot x^{m}\right]\\&=\frac{a}{c}\left(1+h(x)\cdot x^{m}\right)
\end{aligned} $$
where
$$h(x)=-\frac{1}{ac(x^{m}\log x+\varphi)+ad\cdot x^{m}},\  \  \  \lim_{x\rightarrow0}h(x)=-\frac{1}{ac\varphi(0)}\neq0.$$
We could choose a sequence\ $\{x_{n}\}\to 0$ and $\alpha<\beta$ such that
\ $$  \arg{h(x_{n})}\in(\alpha,\,\beta),\quad  \arg {x_{n}} \in \bigg(\frac{-\beta}{m},\, \frac{-\alpha}{m}\bigg)\quad
\text{and}\quad 0<\beta-\alpha<\frac{\pi}{2}.$$
Then we have
\ $$\arg \Big(h(x_{n})\cdot x_{n}^{m}\Big)\in \bigg(-\frac{\pi}{2},\, \frac{\pi}{2}\bigg)\quad \text{and}\quad
\Re \big( h(x_{n})x_{n}^{m}\big)>0.$$
 Recalling the expression of\ $F$ and $|c|=|a|$, we have\ $|F(x_{n})|>1$, which contradicts that\ $F$ takes values in\ $\mathbb{D}$.

Summing up the above two paragraph, we have proved the first statement of the lemma near $p_i$.

We are now going to show there exists a neighborhood\ $U_{j}$ of\ $q_{j}$\ with complex coordinate\ $z$\ and some\ $\mathfrak{L}_{j}\in \textup{PGL}(2,\mathbb{C})$ such that\ $z(q_{j})=0$\ and\ $G_{j}=\mathfrak{L}_{j}\circ F$\ has the form\ $\log z$. Since\ $F$ is compatible with\ $\textup{D}$, we could choose a neighborhood\ $U_{j}$ of\ $q_{j}$\ and a complex coordinate\ $x$\ on\ $U_{j}$\ such that\ $x(q_{j})=0$\ and
\[
                \{F,x\}=\frac{1}{2x^{2}}+\frac{d_{j}}{x}+\phi_{j}(x),
\]
where\ $\phi_{j}(x)$\ is holomorphic in\ $U_{j}$. Analogue to the proof above, the indicial equation
\[
             f(s)=s(s-1)+\frac{1}{4}
\]
has roots\ $s_{1}=s_{2}=1/2$. The following expression
$$u_{\log}=u(s_{2},x)\log x+ x^{s_{2}}\sum_{k\geq 0}c_{k}'(s_{2})x^{k}$$
is a solution of\ $L_{j}u=0$, thus\ $u_{\log}$,\ $u(s_{2},x)$\ are the two linearly independent solutions of\ $L_{j}u=0$\ near\ $0$. We have
$$ F(x)=\frac{au_{\log}+bu(s_{2},x)}{cu_{\log}+du(s_{2},x)}=\begin {pmatrix} a & b \\ c & d \end {pmatrix}\circ \frac{u_{\log}}{u(s_{2},x)},$$
where
$$\begin {pmatrix} a & b \\ c & d \end {pmatrix}\in \textup{PGL}(2,\mathbb{C}), \  \ \frac{u_{\log}}{u(s_{2},x)}=\log x+\frac{\sum_{k=0}^\infty c_{k}'(s_{2})x^{k}}{\sum_{k=0}^\infty c_{k}(s_{2})x^{k}}=\log x+\psi(x),$$
\ $\psi(x)$ is holomorphic near\ $0$ and\ $\psi(0)=0$. Put\ $z=xe^{\psi(x)}$, then\ $\log z=\log x+\psi(x)$,
\ $F=\frac{a\log z+b}{c\log z+d}$.

By now, we have proved the first statement of the lemma. As long as the second statement is concerned,
since\ $F$\ is locally univalent on\ $\Sigma^{*}$ and has monodromy belonging to\ $\textup{{PSU}(1,1)}$,\ $F^{*}\mathrm{d} \sigma_{0}^2$\ is a well-defined smooth conformal hyperbolic metric on\ $\Sigma\backslash {\rm supp}\, D$. The first statement above proved just now implies that this metric has conical singularities at\ $p_{i}$\ with cone angle\ $2\pi\alpha_{i}$\ and cusp singularities at\ $q_{j}$.
\end{proof}

\begin{rem}
\label{rem:l>1}
  {\rm There exists an analogue of this lemma on the upper half-plane model $\mathbb{H}$, which is stated as follows and will be used next section.}

  Let\ $F:\Sigma\backslash {\rm supp}\, D\longrightarrow \mathbb{H}$ be a projective function compatible with
  $\textup{D}=\sum\limits_{i=1}^n(\alpha_{i}-1)p_{i}-\sum\limits_{j=1}^mq_{j}$
   such that the the monodromy of\ $F$\ lies in\ $\textup{PSL}(2,\mathbb{R})$. Then there exists a neighborhood\ $U_{i}$ of\ $p_{i}$\ with complex coordinate\ $z$\ and\ $\mathfrak{L}_{i}\in \textup{PGL}(2,\mathbb{C})$ such that\ $z(p_{i})=0$\ and\ $G_{i}=\mathfrak{L}_{i}\circ F$\ has the form\ $G_{i}(z)=z^{\alpha_{i}}$; there exists a neighborhood\ $U_{j}$ of\ $q_{j}$\ with complex coordinate\ $z$\ and\ $\mathfrak{L}_{j}\in \textup{PGL}(2,\mathbb{C})$, such that\ $z(q_{j})=0$\ and\ $G_{j}=\mathfrak{L}_{j}\circ F$\ has the form\ $G_{j}(z)=\log z$. Moreover the pullback $F^{*}\mathrm{d} \sigma_{0}^2$ of the hyperbolic metric $\mathrm{d} \sigma_{0}^2=\frac{|dz|^2}{(\Im z)^2}$  on ${\Bbb H}$ by $F$ on $\Sigma\backslash {\rm supp}\, D$ is a conformal hyperbolic metric representing $D$.
\end{rem}

\section{Proof of the main theorem}
\label{sec:Belyi}

\paragraph{}
We at first show the first statement of the theorem.
Let\ $\mathrm{d} \sigma^2$\ be a conformal hyperbolic metric on the punctured disk\ $\bigtriangleup^{\ast}=\{{w\in\mathbb{C}|0<\vert w \vert<1}\}$,\ $\mathrm{d} \sigma^2$\ has a conical singularity at\ $w=0$\ with cone angle\ $2\pi\alpha>0$. We will denote the singularity by\ $p$. By Lemma \ref{lem:pq}, there exists a projective function\ $F$\ from\ $\Delta^{\ast}$\ to the unit disc\ $\mathbb{D}$ such that the monodromy of\ $F$\ belongs to \textup{PSU(1,1)} and\ $\mathrm{d} \sigma^2=F^{*}(\mathrm{d} \sigma_{0}^2)$, where\ $\mathrm{d} \sigma_{0}^2=\frac{4|\mathrm{d} z|^2 }{(1-|z|^2)^2}$\ is the hyperbolic metric on\ $\mathbb{D}$. Thus we obtain the monodromy representation\ $\rho_{F}:\pi_{1}(\triangle^{\ast})\longrightarrow \textup{PSU(1,1)}$ of $F$, where $\pi_{1}(\triangle^{\ast})\cong \mathbb{Z}$. We denote\ $\rho_{F}(e)$\ by\ $\mathfrak{L}$, where\ $e$ means the homotopy class of the closed curve\ $t\mapsto\frac{1}{2}\exp(2\pi\sqrt{-1}t),\ t\in[0,1]$. By Lemma \ref{lem:NotEmp}, there exists a neighborhood\ $U$ of\ $p$\ with complex coordinate\ $\xi$\ and\ $\mathfrak{j}\in \textup{PGL}(2,\mathbb{C})$ such that\ $\xi(p)=0$\ and\ $J=\mathfrak{j}\circ F$\ has the form\ $J(\xi)=\xi^{\alpha}$, where\ $\alpha>0$. We divide the argument into the following two cases.

\begin{itemize}
 \item
Suppose that\ $\alpha$\ is an integer. Then the local monodromy of\ $F$\ at\ $p$\ is trivial. We set\ $n:=\alpha\in\mathbb{Z}_{>1}$. Then
$$F(\xi)=\frac{a\xi^{n}+b}{c\xi^{n}+d},\  \   \begin {pmatrix} a & b \\ c & d \end {pmatrix}\in \textup{PGL}(2,\mathbb{C}).$$
We can set\ $F(0)=0$ without loss of generality, then\ $b=0$,\ $d\neq0$,\ $F(\xi)=\frac{a}{c\xi^{n}+d}\xi^{n}$. Therefore we could choose another complex coordinate\ $z=z(\xi)$ of\ $U$\ under which\ $F=F(z)=z^{n}$. Since\ $z(p)=\xi(p)=0$, by continuity, there exists a neighborhood\ $\Delta_{\varepsilon}=\{{z\in\mathbb{C}|\vert z \vert<\varepsilon}\}\subset U$\ for some\ $\varepsilon>0$, so that, on\ $\Delta_{\varepsilon}$, we have
$$ \mathrm{d} \sigma^2|_{\Delta_{\varepsilon}}=\frac{4n^2\vert z \vert^{2n-2}}{(1-\vert z \vert ^{2n})^2}\vert \mathrm{d} z \vert^2.$$

 \item Suppose that\ $\alpha$\ is not an integer, Note that\ $\rho_{F}(e)=\mathfrak{L}\in \textup{PSU(1,1)}$.
 We divide the argument of the second case into the following three subcases.

  \begin{itemize}
          \item
 If\ $\mathfrak{L}$ is elliptic, then there exists\ $\mathfrak{K}\in \textup{PSU(1,1)}$ such that\ $\mathfrak{K}\circ\mathfrak{L}\circ\mathfrak{K^{-1}}(\xi)=e^{2\pi\sqrt{-1}\theta}\xi$ for some real number\ $\theta$. Since\ $F$ is a developing map of\ $\mathrm{d} \sigma^2$, so is \ $G=\mathfrak{K}\circ F$. Therefore we can choose a function element\ $\mathfrak{g}=\frac{a\xi^{\alpha}+b}{c\xi^{\alpha}+d}$ of\ $G$ with\ $ad-bc=1$\ in an open disk which is near $p$ and does not contain\ $p$. Then
$$ e^{2\pi\sqrt{-1}\theta}\mathfrak{g}= e^{2\pi\sqrt{-1}\theta} \frac{a\xi^{\alpha}+b}{c\xi^{\alpha}+d} = \frac{a e^{2\pi\sqrt{-1}\alpha}\xi^{\alpha}+b}{ce^{2\pi\sqrt{-1}\alpha} \xi^{\alpha}+d}.    $$
This is equivalent to the following equalities holding:
$$\left\{\  \begin{aligned} ac e^{2\pi\sqrt{-1}\alpha}(1-e^{2\pi\sqrt{-1}\theta}) &=0, \\
  (ade^{2\pi\sqrt{-1}\alpha}+bc)-e^{2\pi\sqrt{-1}\theta}(bce^{2\pi\sqrt{-1}\alpha}+ad) &=0, \\
  bd(1-e^{2\pi\sqrt{-1}\theta}) &=0. \end{aligned} \right.$$
Solving the system, we find that either\ $c=b=0$\ or\ $a=d=0$. If\ $a=d=0$, then\ $\mathfrak{g}=\frac{b}{c\xi^{\alpha}}$, but\ $G:\triangle^{\ast}\longrightarrow \mathbb{D}$, contradiction! Thus\ $c=b=0$, that is,\ $\mathfrak{g}(\xi)$\ equals\ $\mu \xi^{\alpha}(\mu\neq0)$,\ $\mathfrak{g}(0)=0$. Therefore we could choose another complex coordinate\ $z$ near\ $p$\ under which\ $\mathfrak{g}(z)= z^{\alpha}$. Since\ $z(p)=\xi(p)=0$, by continuity, there exists a neighborhood\ $\Delta_{\varepsilon}=\{{z\in\mathbb{C}|\vert z \vert<\varepsilon}\}$\ for some\ $\varepsilon>0$, so that, on\ $\Delta_{\varepsilon}$, we have
$$ \mathrm{d} \sigma^2|_{\Delta_{\varepsilon}}=\frac{4\alpha^2\vert z \vert^{2\alpha-2}}{(1-\vert z \vert ^{2\alpha})^2}\vert \mathrm{d} z \vert^2. $$

  \end{itemize}

In order to analyze parabolic and hyperbolic transformations, it will be more convenient to work in the upper half-plane model\ $\mathbb{H}$ of the hyperbolic plane than in\ $\mathbb{D}$.
There exists a projective function\ $F:\triangle^{\ast}\longrightarrow \mathbb{H}$\ on\ $\triangle^{\ast}$\ such that the monodromy of\ $F$\ belongs to\ $\textup{PSL}(2,\mathbb{R})$ and\ $\mathrm{d} \sigma^2=F^{*}(\mathrm{d} \sigma_{0}^2)$, where\
$\mathrm{d} \sigma_{0}^2=\frac{|\mathrm{d} z|^2 }{(\Im z)^{2}}$\ is the hyperbolic metric on\ $\mathbb{H}$. Thus we obtain the monodromy representation\ $\rho_{F}:\pi_{1}(\triangle^{\ast})\longrightarrow \textup{PSL}(2,\mathbb{R})$ of $F$,\ $\pi_{1}(\triangle^{\ast})\cong \mathbb{Z}$. We denote\ $\rho_{F}(e)$\  by\ $\mathfrak{L}$.

       \begin{itemize}
         \item
If\ $\mathfrak{L}$ is hyperbolic, then there exists\ $\mathfrak{K}\in \textup{PSL}(2,\mathbb{R})$ such that\ $\mathfrak{K}\circ\mathfrak{L}\circ\mathfrak{K^{-1}}(\xi)=\lambda\xi$ for some positive real number\ $\lambda$. Since\ $F$ is a developing map of\ $\mathrm{d} \sigma^2$, so is \ $G=\mathfrak{K}\circ F$.
Therefore we can choose a function element\ $\mathfrak{g}=\frac{a\xi^{\alpha}+b}{c\xi^{\alpha}+d}$ of\ $G$ with\ $ad-bc=1$\ in an open disk which is near $p$ and does not contain\ $p$. Then
$$\lambda\mathfrak{g}= \lambda \frac{a\xi^{\alpha}+b}{c\xi^{\alpha}+d} = \frac{a e^{2\pi\sqrt{-1}\alpha}\xi^{\alpha}+b}{ce^{2\pi\sqrt{-1}\alpha} \xi^{\alpha}+d}.  $$
This is equivalent to the following equalities holding:
$$\left\{\  \begin{aligned} ac e^{2\pi\sqrt{-1}\alpha}(1-\lambda) &=0, \\
  (ade^{2\pi\sqrt{-1}\alpha}+bc)-\lambda(bce^{2\pi\sqrt{-1}\alpha}+ad) &=0, \\
  bd(1-\lambda) &=0. \end{aligned} \right.$$
Since\ $\lambda$\ is a positive real number,\ $\lambda\neq1$ and\ $ad-bc=1$, there is no solution to this system, contradiction!

         \item
 If\ $\mathfrak{L}$ is parabolic, then there exists\ $\mathfrak{K}\in \textup{PSL}(2,\mathbb{R})$ such that\ $\mathfrak{K}\circ\mathfrak{L}\circ\mathfrak{K^{-1}}(\xi)=\xi+t$ for some real non-zero number\ $t$. Since\ $F$ is a developing map of\ $\mathrm{d} \sigma^2$, so is \ $G=\mathfrak{K}\circ F$.
Therefore we can choose a function element\ $\mathfrak{g}=\frac{a\xi^{\alpha}+b}{c\xi^{\alpha}+d}$ of\ $G$ with\ $ad-bc=1$\ in an open disk which is near $p$ and does not contain\ $p$. Then
$$ \mathfrak{g}+t= \frac{a\xi^{\alpha}+b}{c\xi^{\alpha}+d}+t= \frac{a e^{2\pi\sqrt{-1}\alpha}\xi^{\alpha}+b}{ce^{2\pi\sqrt{-1}\alpha} \xi^{\alpha}+d}.    $$
This is equivalent to the following equalities holding:
$$\left\{\  \begin{aligned} tc^{2} e^{2\pi\sqrt{-1}\alpha} &=0, \\
  ad+tcd-bc+(bc+tcd-ad)e^{2\pi\sqrt{-1}\alpha} &=0, \\
  td^{2} &=0. \end{aligned} \right.$$
Since\ $t\neq0$, we have\ $c=d=0$, but\ $ad-bc=1$, contradiction!

    \end{itemize}
 \end{itemize}

We show the uniqueness of the complex coordintae $z$ in the first statement of the theorem. Let\ $z$ and\ $\widetilde{z}$\ be coordinates such that conditions of the theorem are satisfied, then\ $F(z)=z^{\alpha}$,\ $\widetilde{F}(\widetilde{z})=\widetilde{z}^{\alpha}$\ are all developing maps of\ $ \mathrm{d} \sigma^2$. By lemma \ref{lem:contr}, there exists\ $\mathfrak{L}\in \textup{PSU(1,1)}$ such that\ $\widetilde{F}=\mathfrak{L}\circ F$, then\ $\widetilde{F}=\frac{aF+b}{\overline{b}F+\overline{a}}\ ,\ \ |a|^{2}-|b|^{2}=1$. Since\ $z(p)=\widetilde{z}(p)=0$, $F(p)=\widetilde{F}(p)=0$, we have\ $b=0$  by a calculation. Thus\ $\widetilde{F}=\frac{a}{\overline{a}}F=\mu F ,\ |\mu|=1$, then there exists an open disk\ $V$\ which is near $p$ and does not contain\ $p$ such that\ $\widetilde{z}^{\alpha}=\mu z^{\alpha}$. Therefore we have\ $\widetilde{z}=\lambda z$ on\ $V$ with\ $|\lambda|=1$. Since\ $z$,\ $\widetilde{z}$ and\ $w$ are coordinates near\ $p$,\ $z$ and\ $\widetilde{z}$\ are holomorphic functions of\ $w$, then\ $\widetilde{z}=\lambda z$,\ $|\lambda|=1$\ holds in a neighborhood of\ $p$.\\

What follows is the proof for the second statement of the theorem.
Let\ $\mathrm{d} \sigma^2$\ be a conformal hyperbolic metric on the punctured disk\ $\bigtriangleup^{\ast}=\{{w\in\mathbb{C}|0<\vert w \vert<1}\}$,\ $\mathrm{d} \sigma^2$\ has a cusp singularity at\ $w=0$. We will denote the singularity by\ $q$. By Lemma \ref{lem:pq}, there exists a projective function\ $F$\ from\ $\Delta^{\ast}$\ to the unit disc\ $\mathbb{D}$ such that the monodromy of\ $F$\ belongs to \textup{PSU(1,1)} and\ $\mathrm{d} \sigma^2=F^{*}(\mathrm{d} \sigma_{0}^2)$, where\
$\mathrm{d} \sigma_{0}^2=\frac{4|\mathrm{d} z|^2 }{(1-|z|^2)^2}$\ is the hyperbolic metric on\ $\mathbb{D}$. Thus we obtain the monodromy representation\ $\rho_{F}:\pi_{1}(\triangle^{\ast})\longrightarrow \textup{PSU(1,1)}$ of $F$,\ $\pi_{1}(\triangle^{\ast})\cong \mathbb{Z}$. We denote\ $\rho_{F}(e)$\  by\ $\mathfrak{L}$. By Lemma \ref{lem:NotEmp}, there exists a neighborhood\ $U$ of\ $q$\ with complex coordinate\ $\xi$\ and\ $\mathfrak{j}\in \textup{PGL}(2,\mathbb{C})$ such that\ $\xi(q)=0$\ and\ $J=\mathfrak{j}\circ F$\ has the form\ $J(\xi)=\log \xi$. The argument is divided into the following three cases.

 \begin{itemize}
          \item
 If\ $\mathfrak{L}$ is elliptic, then there exists\ $\mathfrak{K}\in \textup{PSU(1,1)}$ such that\ $\mathfrak{K}\circ\mathfrak{L}\circ\mathfrak{K^{-1}}(\xi)=e^{2\pi\sqrt{-1}\theta}\xi$ for some real number\ $\theta$. Since\ $F$ is a developing map of\ $\mathrm{d} \sigma^2$, so is\ $G=\mathfrak{K}\circ F$.
Therefore we can choose a function element\ $\mathfrak{g}=\frac{a\log\xi+b}{c\log\xi+d}$\ of\ $G$ with\ $ad-bc=1$\ in an open disk which is near $q$ and does not contain\ $q$. Then
$$ e^{2\pi\sqrt{-1}\theta}\mathfrak{g}= e^{2\pi\sqrt{-1}\theta} \frac{a\log\xi+b}{c\log\xi+d} = \frac{a (\log\xi+2\pi\sqrt{-1})+b}{c(\log\xi+2\pi\sqrt{-1})+d}.    $$
This is equivalent to the following equalities holding:
 $$\left\{\  \begin{aligned} ac(1-e^{2\pi\sqrt{-1}\theta}) &=0, \\
  (ac2\pi\sqrt{-1}+ad+bc)(1-e^{2\pi\sqrt{-1}\theta}) &=0, \\
  (ad-e^{2\pi\sqrt{-1}\theta}bc)2\pi\sqrt{-1}+(1-e^{2\pi\sqrt{-1}\theta})bd &=0. \end{aligned} \right.$$
Solving the system, we find that either\ $a=b=0$\ or\ $c=d=0$. But\ $ad-bc=1$, contradiction!

  \end{itemize}

  \begin{itemize}
         \item
 If\ $\mathfrak{L}$ is hyperbolic, then there exists\ $\mathfrak{K}\in \textup{PSL}(2,\mathbb{R})$ such that\ $\mathfrak{K}\circ\mathfrak{L}\circ\mathfrak{K^{-1}}(\xi)=\lambda\xi$ for some positive real number\ $\lambda$. Since\ $F$ is a developing map of\ $\mathrm{d} \sigma^2$, so is\ $G=\mathfrak{K}\circ F$.
Therefore we can choose a function element\ $\mathfrak{g}=\frac{a\log\xi+b}{c\log\xi+d}$ of\ $G$ with\ $ad-bc=1$\ in an open disk which is near $q$ and does not contain\ $q$. Then
$$ \lambda\mathfrak{g}= \lambda \frac{a\log\xi+b}{c\log\xi+d} = \frac{a(\log\xi+2\pi\sqrt{-1})+b}{c(\log\xi+2\pi\sqrt{-1})+d}.    $$
This is equivalent to the following equalities holding:
$$\left\{\  \begin{aligned} ac(1-\lambda) &=0, \\
  (ac2\pi\sqrt{-1}+ad+bc)(1-\lambda) &=0, \\
  \lambda b(c2\pi\sqrt{-1}+d)-d(a2\pi\sqrt{-1}+b) &=0. \end{aligned} \right.$$
Note that\ $\lambda$\ is a positive real number and\ $\lambda\neq1$. Solving the system, we find that either\ $a=b=0$\ or\ $c=d=0$. But\ $ad-bc=1$, contradiction!

         \item
 If\ $\mathfrak{L}$ is parabolic, then there exists\ $\mathfrak{K}\in \textup{PSL}(2,\mathbb{R})$ such that\ $\mathfrak{K}\circ\mathfrak{L}\circ\mathfrak{K^{-1}}(\xi)=\xi+t$ for some real non-zero number\ $t$. Since\ $F$ is a developing map of\ $\mathrm{d} \sigma^2$, so is\ $G=\mathfrak{K}\circ F$.
Therefore we can choose a function element\ $\mathfrak{g}=\frac{a\log\xi+b}{c\log\xi+d}$ of\ $G$ with\ $ad-bc=1$\ in an open disk which is near $q$ and does not contain\ $q$. Then
$$ \mathfrak{g}+t= \frac{a\log\xi+b}{c\log\xi+d}+t = \frac{a(\log\xi+2\pi\sqrt{-1})+b}{c(\log\xi+2\pi\sqrt{-1})+d}.    $$
This is equivalent to the following equalities holding:
$$\left\{\  \begin{aligned} tc^{2} &=0, \\
  2tc(c\pi\sqrt{-1}+d) &=0, \\
  td(c2\pi\sqrt{-1}+d)-2\pi\sqrt{-1} &=0. \end{aligned} \right.$$
Sinc\ $t\neq0$, we have\ $c=0$, $td^{2}=2\pi\sqrt{-1}$. By\ $ad-bc=1$, then\ $ad=1$,\ $t=a^{2}2\pi\sqrt{-1}$. Since\ $t$ is a real non-zero number, put\ $a^{2}=\sqrt{-1}\delta$£¬ where\ $\delta$ is a real non-zero number. Moreover, \ $\mathfrak{g}(\xi)=a^{2}\log\xi+ab=a^{2}(\log\xi+b/a)$, set\ $z=\xi e^{b/a}$, then\ $\log z=\log\xi+b/a$, \ $\mathfrak{g}(z)=a^{2}\log z=\sqrt{-1}\delta\log z$. If\ $\delta>0$, put\ $H= \begin {pmatrix} \frac{1}{\sqrt{\delta}} & 0 \\ 0 & \sqrt{\delta} \end {pmatrix}\circ G$, then\ $H(z)=\sqrt{-1}\log z$ should be a developing map, contradicts to\ $H:\triangle^{\ast}\longrightarrow \mathbb{H}$. Thus,\ $\delta<0$, put\ $H= \begin {pmatrix} \frac{1}{\sqrt{-\delta}} & 0 \\ 0 & \sqrt{-\delta} \end {pmatrix}\circ G$, then\ $H(z)=-\sqrt{-1}\log z$ is a developing map. Since\ $z(q)=\xi(q)=0$, by continuity, there exists a neighborhood\ $\Delta_{\varepsilon}=\{{z\in\mathbb{C}|\vert z \vert<\varepsilon}\}$\ for some\ $\varepsilon>0$, so that, on\ $\Delta_{\varepsilon}$, we have
$$ \mathrm{d} \sigma^2|_{\Delta_{\varepsilon}}=\vert z \vert ^{-2}(\ln|z|)^{-2}|dz|^{2}.   $$

 \end{itemize}

Finally, we show the uniqueness of the complex coordinate $z$ in the second statement of the theorem. Let\ $z$ and\ $\widetilde{z}$\ be coordinates such that conditions of the theorem are satisfied, then\ $F(z)=-\sqrt{-1}\log z$,\ $\widetilde{F}(\widetilde{z})=-\sqrt{-1}\log \widetilde{z}$ are all developing maps of\ $\mathrm{d} \sigma^2$. By lemma \ref{lem:contr}, there exists\ $\mathfrak{L}\in \textup{PSL}(2,\mathbb{R})$ such that\ $\widetilde{F}=\mathfrak{L}\circ F$, then\ $\widetilde{F}=\frac{aF+b}{cF+d}\ ,\ ad-bc=1$. Since\ $z(q)=\widetilde{z}(q)=0$,\ $F(q)=\widetilde{F}(q)=\infty$, we have\ $c=0$  by a calculation. Thus\ $\widetilde{F}=\frac{aF+b}{d}=a^{2}F+ab$, then\ $-\sqrt{-1}\log\widetilde{z}=-a^{2}\sqrt{-1}\log z+ab$, there exists an open disk\ $V$\ which is near $q$ and does not contain\ $q$\ such that\ $a^{2}=1$,\ $\log\widetilde{z}=\log z+ab\sqrt{-1}$. Therefore we have\ $\widetilde{z}=\lambda z$ on\ $V$ with\ $|\lambda|=1$. Since\ $z$,\ $\widetilde{z}$ and\ $w$ are coordinates near\ $q$,\ $z$ and\ $\widetilde{z}$\ are holomorphic functions of\ $w$, then\ $\widetilde{z}=\lambda z$,\ $|\lambda|=1$\ holds in a neighborhood of\ $q$.

\begin{center}
 {\bf Acknowledgements}
\end{center}

\noindent Xu would like to express his sincere gratitude to Professor Song Sun at University of California, Berkeley for the stimulating conversations in the summer of 2016, which motivated this manuscript.
Xu is supported in part by the National Natural Science Foundation of China (grant no. 11571330), and both Shi and Xu are supported in part by
the Fundamental Research Funds for the Central Universities.




\end{document}